\documentclass[12pt]{amsart}
\usepackage{amsmath,amsthm,cite,graphicx}
\usepackage{amsbsy,amsfonts,amssymb,mathtools}
\usepackage{amsaddr, ytableau,mathtools}
\def\ds{\displaystyle}

\def\be{{\bf e}}
\def\ov{\overline}
\def\byt{\begin{ytableau}}
\def\eyt{\end{ytableau}}

\def\lb{\left(}
\def\rb{\right)}

\newtheorem{thm}{Theorem}[section]
\newtheorem{lem}{Lemma}[section]

\newtheorem{cor}{Corollary}[section]

\newtheorem{exa}{Example}[section]
\newtheorem{rem}{Remark}[section]

\newdimen\Squaresize \Squaresize=14pt
\newdimen\Thickness \Thickness=0.4pt

\def\scc{\scriptstyle}

\def\Square#1{\hbox{\vrule width \Thickness
   \vbox to \Squaresize{\hrule height \Thickness\vss
      \hbox to \Squaresize{\hss#1\hss}
   \vss\hrule height\Thickness}
\unskip\vrule width \Thickness} \kern-\Thickness}

\def\Vsquare#1{\vbox{\Square{$#1$}}\kern-\Thickness}

\def\young#1{
\vbox{\smallskip\offinterlineskip \halign{&\Vsquare{##}\cr #1}}}
\def\vy#1{\hskip2pt\vcenter{\young{#1}}\hskip2pt}

\newcommand{\shiftnoarrow}[2]{\ensuremath \raisebox{#1cm}{${#2}$}}

\title[A cyclic sieving phenomenon for symplectic tableaux]%
{A cyclic sieving phenomenon for symplectic tableaux}

\author{Graeme Henrickson, Anna Stokke and
Max Wiebe}
\address{University of Winnipeg \\
Department of Mathematics and Statistics \\
Winnipeg, Manitoba \\
Canada  R3B 2E9}

\email{\tt a.stokke@uwinnipeg.ca}
\thanks{This research was supported by NSERC grant RGPIN-2018-05877.}

\begin{document}

\begin{abstract}
We give a  cyclic sieving phenomenon for symplectic $\lambda$-tableaux $SP(\lambda,2m)$, where $\lambda$ is a partition of an odd integer $n$ and $gcd(m,p)=1$ for any odd prime $p\leq n$.  We use the crystal structure on  Kashiwara-Nakashima symplectic tableaux to get a cyclic sieving action as the product $\sigma$ of simple reflections in the Weyl group.  The cyclic sieving polynomial is the $q$-anologue of the hook-content formula for symplectic tableaux.  More generally, we give a CSP for symplectic skew tableaux with analogous conditions on the shape and a cyclic group action that rotates tableaux weights in a way motivated by the $\sigma$-action.
\end{abstract}

\keywords{Cyclic sieving phenomenon; Crystal graphs; Young tableau; Symplectic tableaux}

\maketitle

\section{Introduction}

The {\em cyclic sieving phenomenon} (CSP) was introduced by Reiner, Stanton, and White in \cite{rsw}.  Let $X$ be a finite set, $\langle g \rangle$ a cyclic group of order $n$ that acts on $X$ and $f(q) \in \mathbb{Z}[q]$.    The triple $(X, \langle g \rangle,f(q))$ exhibits the cyclic sieving phenomenon if, for $\omega$ a primitive $n$th root of unity, $$\vert \{ x \in X \mid g^d\cdot x=x \} \vert =f(\omega^d),$$ for all $d\geq 0$. Since then, CSPs have been widely studied in various settings. For a 2011 survey see \cite{sagan}.

Numerous researchers have investigated cyclic sieving phenomena for tableaux (see, for instance, \cite{alexander, linusson,basman, bms, fontaine, gaetz, ohpark1, ohpark2, pechenik, ponwang, psv, rhoades, westbury1}).  Using the cyclic action given by Sch\"utzenberger's promotion  operator $\partial$ \cite{schutz1, schutz2} on rectangular semistandard tableaux $SSYT(\lambda,m)$ with entries in $\{1,2,\ldots,m\}$, Rhoades proved that the triple
$$(SSYT(\lambda,m), \langle \partial \rangle, q^{-\kappa(\lambda)}s_\lambda(1,q,\ldots,q^{m-1}))$$
exhibits the CSP.  Here $s_\lambda(1,q,\ldots,q^{m-1})$ is a principal specialization of the Schur polynomial and $\kappa(\lambda)=\sum_{i}(i-1)\lambda_i$.
The above result was also proved in \cite{westbury1} using crystal base theory.  As well, CSPs have been given for hook shapes  \cite{bms} and for stretched hook shapes  \cite{linusson}.    In \cite{ohpark2}, the authors show that a CSP can be found for Schur polynomials and more general shapes, but the group action is unknown.

Given a finite-dimensional simple complex Lie algebra $\mathfrak{g}$ with irreducible highest weight $U_q(\mathfrak{g})$-module $V_q(\lambda)$, the crystal base $\mathcal{B}(\lambda)$ reflects the structure of the $U_q(\mathfrak{g})$-module $V_q(\lambda)$ in a combinatorial way and so reveals information about the structure of the irreducible highest weight $\mathfrak{g}$-module $V(\lambda)$.  Kashiwara and Nakashima gave Young tableaux realizations of crystal bases $\mathcal{B}(\lambda)$ for classical simple Lie algebra types in \cite{kashnak}.  Essentially this is a description of each $\mathcal{B}(\lambda)$ as a set of tableaux of shape $\lambda$ satisfying certain conditions. For Cartan type $A_{m-1}$, these are the usual semistandard tableaux, but for other Cartan types the tableaux descriptions are  more complicated.

Oh and Park \cite{ohpark1} employed the cyclic action $\tt{c}$ arising from the $U_q(\mathfrak{sl}_m)$-crystal structure for semistandard tableaux to prove that  $$(SSYT(\lambda,m),\langle \mathtt{c} \rangle, q^{-\kappa(\lambda)}s_\lambda(1,q,\ldots,q^{m-1}))$$
exhibits the CSP when $\ell(\lambda)$ (the length of $\lambda$) is less than $m$ and $\mbox{gcd}(m,\vert \lambda \vert)=1$.  This result was extended to skew shapes in \cite{alexander}.
The action arising from the crystal structure of a $U_q(\mathfrak{g})$-module was further studied in \cite{ohpark2}.  When $\mathfrak{g}$ is type $A_{m-1}$, $\ell(\lambda)<m$, and there is at least one fixed point under the action of $\mathtt{c}$, they showed that $(SSYT(\lambda,m),\langle {\tt c} \rangle, s_{\lambda}(1,q,q^{2},\ldots,q^{m-1}))$ exhibits the CSP if and only if $\lambda=(am)^b$ where $b=1$ or $b=m-1$.

In this paper, we consider CSPs for symplectic tableaux.  There are a few different types of symplectic tableaux, which index bases for irreducible $\mathfrak{sp}(2m)$-modules $V(\lambda)$, where $\ell(\lambda) \leq m$ (see \cite{sundaram} or \cite{stokkesymplectic}).    Symplectic King tableaux (see \cite{king, kingelsh}) are relatively easy to describe.  De Concini described a different version of symplectic tableaux in \cite{deconcini} and Sheats gave a weight-preserving bijection between De Concini and King tableaux in \cite{sheats}.    Kashiwara and Nakashima described symplectic tableaux endowed with a crystal structure in \cite{kashnak}.  These are related to De Concini tableaux through a straighforward bijection.

The highest weight $U_q(\mathfrak{sp}(2m))$-crystal $\mathcal{B}(\lambda)$ with highest weight $\lambda$ can be realized as the set of Kashiwara-Nakashima symplectic tableaux $SP(\lambda,2m)$ of shape $\lambda$.  The  crystal structure  leads to a cyclic action on $SP(\lambda,2m)$ given by the product $\sigma=\sigma_1\sigma_2 \cdots \sigma_m$ of simple reflection operators in the Weyl group and $G=\langle \sigma \rangle$ has order $2m$.

We prove our main results in the general situation for KN-skew symplectic tableaux $SP(\lambda/\mu,2m)$ coupled with any cyclic group action, with group order $2m$, such that a generator takes tableaux weights $(\chi_1,\ldots,\chi_m)$ to $(-\chi_m,\chi_1,\ldots,\chi_{m-1})$.  When the skew shape $\lambda/\mu$ has $n$ boxes, where $n$ is odd, and $gcd(m,p)=1$ for any odd prime $p\leq n$,  we prove that every orbit under such an action has order $2m$ in Theorem \ref{orderthm}.  As a corollary, this holds for $SP(\lambda,2m)$ with the action of $\langle \sigma \rangle$.

As is the case for semistandard tableaux, there is a hook-content formula that counts the number of symplectic tableaux (see \cite{campbellstokke,kingelsamra}) and we use its $q$-analogue $f_{sp}^\lambda(q)$ to give a CSP for $SP(\lambda,2m)$. We give a nice form for $f_{sp}^\lambda(q)$ in Section \ref{sec:hooks}.  Next, we partition weights into sets of size $2m$ using an action on the weights induced by the subgroup of $\mathfrak{S}_m$ corresponding to the dihedral group of order $2m$.  This allows us to give an appropriate form in Theorem \ref{equivthm}  for the polynomial $X(q)$ that we use to give a CSP for $SP(\lambda/\mu,2m)$. When $\mu=\emptyset$, $X(q)=q^{\kappa(\lambda)}f_{sp}^\lambda(q)$.

When $\vert \lambda/\mu \vert$ is odd and $gcd(m,p)=1$ for odd primes $p\leq n$, $SP(\lambda/\mu,2m)$, with a cyclic group action satisfying the properties described above, and polynomial $X(q)$ gives a CSP, which we prove in Theorem \ref{cspthm}.  As a corollary, for $\lambda$ a partition of an odd integer $n$ and  $gcd(m,p)=1$ for any odd prime $p\leq n$, the following is a CSP-triple: $$(SP(\lambda,2m),\langle \sigma \rangle, f_{sp}^\lambda(q)).$$  In \cite{pappe}, the authors prove another new CSP for Cartan type $C$.  They prove a CSP  for the set of highest weight elements of weight zero in the $n$-fold tensor power of the type $C_m$ crystal.  In \cite{prw}, the authors  gave a correspondence in this setting between the highest weight elements of weight zero and chord diagrams that intertwines promotion and rotation.

We begin the paper with a review of crystal base theory, with a particular focus on Cartan type $C_m$.  Next, we discuss Kashiwara-Nakashima tableaux and the associated crystal action in Section \ref{sec:tableaux}.  In Section \ref{sec:hooks} we prove results concerning the $q$-analogue of the symplectic hook-content formula.   Section \ref{sec:main} is devoted to our main results, where we prove our cyclic sieving phenomenon.

\section{Crystal bases}\label{sec:crystals}

In this section, we review crystal base theory.  For an introduction to Lie algebras, the reader is referred to \cite{erdmann} or \cite{humphreys}.   For a more thorough coverage of crystal bases, see \cite{bump} and \cite{hongkang}.

Let $\mathfrak{g}$ be a finite-dimensional simple complex Lie algebra and let $U_q(\mathfrak{g})$ be its quantum group.  Let $\Phi$ be its root system, with index set $I$, weight lattice $\Lambda$ and simple roots $\{\alpha_i \mid i \in I\}$.  The co-root of $\alpha \in \Phi$ is $\ds \alpha^{\vee}=\frac{2\alpha}{(\alpha, \alpha)}$.

We can associate a {\em Kashiwara crystal} ({\em crystal} for short)  to the root system.  This is a set $\mathcal{B}$ together with maps $\mbox{wt}: \mathcal{B} \rightarrow \Lambda$, $e_i$, $f_i: \mathcal{B} \rightarrow \mathcal{B} \sqcup \{0\}$ and $\epsilon_i$, $\phi_i:\mathcal{B} \rightarrow \mathbb{Z} \sqcup \{-\infty\}$ satisfying the following properties:
\begin{enumerate}
\item $f_i(b)=b^\prime$ if and only if $b=e_i(b^\prime)$ for all $b,b^\prime \in \mathcal{B}$, $i \in I$;
\item $\mbox{wt}(e_i(b))=\mbox{wt}(b)+ \alpha_i$, if $e_i(b) \in \mathcal{B}$, and $\mbox{wt}(f_i(b))=\mbox{wt}(b)- \alpha_i$, if $f_i(b) \in \mathcal{B}$;
\item $\epsilon_i(e_i(b))=\epsilon_i(b)-1$ and $\phi_i(e_i(b))=\phi_i(b)+1$ if $e_i(b) \in \mathcal{B}$;
\item $\epsilon_i(f_i(b))=\epsilon_i(b)+1$ and $\phi_i(f_i(b))=\phi_i(b)-1$ if $f_i(b) \in \mathcal{B}$;
\item $\phi_i(b)=\epsilon_i(b)+( \mbox{wt}(b), \alpha_i^\vee )$ for all $i \in I$;
\item If $\phi_i(b)=-\infty$ for $b \in B$, then $e_i(b)=f_i(b)=0.$

\end{enumerate}

The {\em crystal graph} of $\mathcal{B}$ is a directed graph, which is given  by taking $\mathcal{B}$ as the set of vertices and defining an edge $b \overset{i}\rightarrow b^\prime$ if and only if $f_i(b)=b^\prime$ for $i \in I$.

There is a crystal $\mathcal{B}(\lambda)$ associated to each irreducible highest weight $U_q(\mathfrak{g})$-module $V_q(\lambda)$ that reflects its structure. Kashiwara and Nakashima gave Young tableaux realizations of crystal bases $\mathcal{B}(\lambda)$ for classical simple Lie algebra types in \cite{kashnak}.
In this paper we focus on Cartan type $C_m$ and, unless stated otherwise, when we refer to the crystal $\mathcal{B}(\lambda)$, we mean the crystal of the irreducible highest weight
$U_q(\mathfrak{sp}(2m))$-module with highest weight $\lambda$.

\begin{exa} The type $C_m \ (m \geq 2)$ finite-dimensional Lie algebra can be realized as the symplectic Lie algebra $\mathfrak{sp}(2m, \mathbb{C})$.  If ${\be_i}=(0,\ldots,1, \ldots,0)$ denotes the unit vector with one in the $i$th position, then $$\Phi=\{\pm \be_i \pm \be_j \mid i <j\} \cup \{\pm 2\be_i\},$$ and the set of positive roots are
$$\Phi^+=\{{\be_i} \pm \be_j \mid i<j\} \cup \{2\be_i\}.$$
The weight lattice is $\Lambda=\mathbb{Z}^m$ and a weight $\lambda=(\lambda_1,\ldots, \lambda_m)$ is dominant if and only if $\lambda_1 \geq \lambda_2 \geq \cdots \geq \lambda_m \geq 0$.  Let $\alpha_i=\be_i-\be_{i+1}$, for $1 \leq i \leq m-1$ and let $\alpha_m=2\be_m$.  Then $\{\alpha_1,\ldots,\alpha_{m-1},\alpha_m\}$ is the set of simple roots, which is a basis for  $\Phi$, and the {\em fundamental weights} are $\omega_i=\be_1+\be_2+\cdots + \be_i$, $1 \leq i \leq m$.  The Weyl group for  $\mathfrak{sp}(2m)$ is the hyperoctahedral group, which is the group of signed permutations $\pi$ of $\{\pm 1,\pm 2,\ldots, \pm m\}$, where $\pi(-i)=-\pi(i)$ for $1 \leq i \leq m$.  

\bigskip

The standard $C_m$-crystal  $\mathcal{B}(1)$ has crystal graph and crystal operator as follows:
$$\vy{1\cr}  \overset{\scc 1}\longrightarrow \vy{2 \cr}  \overset{\scc 2}\longrightarrow \cdots  \overset{\scc m-1}\longrightarrow \vy{m \cr}  \overset{\scc m}\longrightarrow \vy{\ov{m}\cr}  \overset{\scc m-1}\longrightarrow  \cdots  \overset{\scc 2}\longrightarrow \vy{\ov{2} \cr}  \overset{\scc 1}\longrightarrow \vy{\ov{1}\cr}$$

$$f_i\left(\vy{j \cr}\right)=\begin{cases}\vy{\scc i+1 \cr} & \mbox{if } j=i \mbox{ and } 1 \leq i \leq m-1\cr
\vy{\ov{i} \cr}  & \mbox{if } i=j=m \mbox{ or } j=  \ov{i+1} \cr
\ \ 0 & \mbox{otherwise} \cr \end{cases}$$

\noindent As well,  $\mbox{wt}\lb\vy{i\cr}\rb=\be_i$, $\mbox{wt}\lb\vy{\ov{i}\cr}\rb =-\be_i$, $\phi_i(x)=\mbox{max}\{ k \in \mathbb{Z}_{\geq 0} \mid f_i^k(x) \neq 0\}$ and $\epsilon_i(x)=\mbox{max}\{k \in \mathbb{Z}_{\geq 0} \mid e_i^k(x) \neq 0\}$.\end{exa}

The tensor product $\mathcal{B} \otimes \mathcal{C}$ of two crystals with the same underlying root system has a crystal structure with
$\mbox{wt}(x \otimes y)=\mbox{wt}(x) + \mbox{wt}(y),$ for $x \in \mathcal{B}$, $y \in \mathcal{C}$, and  tensor product formula

$$f_i \lb x \otimes y \rb=\begin{cases}f_i\lb x  \rb \otimes  y & \mbox{if } \phi_i\lb y \rb \leq \epsilon_i\lb x \rb \cr
x \otimes f_i\lb y \rb & \mbox{otherwise} \cr
\end{cases}$$

$$e_i \lb x  \otimes  y \rb=\begin{cases} x \otimes e_i \lb y \rb & \mbox{if } \phi_i \lb y \rb \geq \epsilon_i \lb x \rb \cr e_i\lb x \rb \otimes  y  & \mbox{otherwise} \cr

\end{cases},$$
where $\phi_i \lb x  \otimes  y \rb=\phi_i \lb x \rb + \mbox{max} ( 0,\phi_i \lb y \rb - \epsilon_i  \lb x \rb )$, $\epsilon_i\lb x  \otimes  y \rb=\epsilon_i \lb y \rb + \mbox{max}(0,\epsilon_i \lb x \rb-\phi_i \lb y \rb).$ (The above coincides with the tensor product rule used in \cite{bump}, but is slightly different than the tensor product rule in \cite{hongkang}.)

The procedure for applying $f_i$ to an element of $\mathcal{B}(1)^{\otimes k}$ can be determined combinatorially using the (symplectic) {\em signature rule}.    If $\vy{x_1\cr} \otimes \vy{x_2 \cr} \otimes \cdots \otimes \vy{x_k \cr} \in \mathcal{B}(1)^{\otimes k}$, then $$f_i \lb\vy{x_1\cr} \otimes \vy{x_2 \cr} \otimes \cdots \otimes \vy{x_k \cr} \rb=\vy{x_1\cr} \otimes \vy{x_2 \cr} \otimes \cdots \otimes f_i\lb \vy{x_j \cr} \rb \otimes \cdots \otimes \vy{x_k \cr}$$ where $x_j$ is determined as follows:

\begin{enumerate}
\item Place a $-$ above $x_s$ if $x_s=i$ or $x_s=\ov{i+1}$ and place a $+$ above $x_s$ if $x_s=\ov{i}$ or $x_s=i+1$.  If every $-$ is left of every $+$ then $x_j$ is equal to the rightmost $x_s$ that is labelled with $-$.
\item Otherwise,  bracket a $+$ with a $-$ to its right so that there are no $+$'s or $-$'s in between.
\item Continue bracketing $+$'s with $-$'s with no unbracketed $+$'s or $-$'s in between until all unbracketed $-$'s are left of unbracketed $+$'s.

\item Choose $x_j$ to be the rightmost unbracketed $-$.  If there are no unbracketed $-$'s, the result is $0$.
\end{enumerate}

\begin{exa}To determine $f_1\lb \vy{1\cr} \otimes \vy{\ov{2} \cr} \otimes \vy{\ov{1}\cr} \otimes \vy{2 \cr} \otimes \vy{\ov{2}\cr} \otimes \vy{1\cr} \rb$, bracket as follows:
$$\overset{-}{\vy{1\cr}} \otimes \overset{-}{\vy{\ov{2} \cr}} \otimes \overset{(+}{\vy{\ov{1}\cr}} \otimes \overset{(+}{\vy{2 \cr}} \otimes \overset{-)}{\vy{\ov{2}\cr}} \otimes \overset{-)}{\vy{1\cr}}, \mbox{ so}$$
 $f_1\lb \vy{1\cr} \otimes \vy{\ov{2} \cr} \otimes \vy{\ov{1}\cr} \otimes \vy{2 \cr} \otimes \vy{\ov{2}\cr} \otimes \vy{1\cr} \rb= \vy{1\cr} \otimes \vy{\ov{1} \cr} \otimes \vy{\ov{1}\cr} \otimes \vy{2 \cr} \otimes \vy{\ov{2}\cr} \otimes \vy{1\cr}.$ \end{exa}
\section{Kashiwara-Nakashima (symplectic) tableaux}

\label{sec:tableaux}

Crystals of tableaux for type $C_m$ are constructed by embedding {\em Kashiwara-Nakashima tableaux} (KN-tableaux) into tensor powers of the standard crystal.

A {\em partition} $\lambda$ of a positive integer $n$ is a $k$-tuple $\lambda=(\lambda_1,\ldots,\lambda_k)$, where $\lambda_1 \geq \lambda_2 \geq \cdots \geq \lambda_k>0$ and $\vert \lambda \vert= \sum_{i=1}^k \lambda_i=n$.  The length of $\lambda$ is $\ell(\lambda)=k$ and the {\em Young diagram} of shape $\lambda$ is given by arranging $n$ boxes in $k$ left-justifed rows with $\lambda_i$ boxes in the $i$th row.  The conjugate of $\lambda$ is the partition $\lambda^t = (\lambda_1^t, \lambda_2^t, . . . , \lambda_r^t)$ where $\lambda_i^t$ is the number of boxes in the $i$th column of the Young diagram of shape $\lambda$.

A {\em semistandard tableau} of shape $\lambda$ is a filling of the Young diagram of shape $\lambda$ with positive integers such that the entries in each row are weakly increasing from left to right and the entries in each column are strictly increasing from top to bottom.  The set  of semistandard tableaux of a given shape $\lambda$  admits a $U_q(\mathfrak{sl}(n))$-crystal structure (see \cite{bump} for details).

The irreducible $\mathfrak{sp}(2m)$-representations are indexed by partitions $\lambda$ with $\ell(\lambda)\leq m$ so we will assume  that $\ell(\lambda) \leq m$.
Semistandard KN-tableaux have entries from the set $\mathcal{M}=\{1, 2, \ldots, m, \ov{m}, \ldots, \ov{1}\}$ with ordering $$1 <2 < \cdots <m < \ov{m} < \cdots <\ov{2} <\ov{1}.$$
A semistandard KN-tableau $T$ of shape $\lambda$ is a filling of the Young diagram of shape $\lambda$ with entries from $\mathcal{M}$ that satisfies the following properties:
\begin{enumerate}
\item The entries in $T$ are weakly increasing across rows from left to right and strictly increasing down columns from top to bottom.
\item For every column in $T$ that contains both an $i$ and an $\ov{i}$, where $i$ belongs to the $p$-th box from the top and $\ov{i}$ belongs to the $q$-th box from the bottom, we have $p+q\leq i$.

\item If $T$ has two adjacent columns having one of the following configurations, where $p,q,r,s$ are the relevant row numbers (where rows are counted from top to bottom), with $p \leq q <r \leq s$ and $i \leq j$, then $(q-p)+(s-r)<j-i$.

$$\left.\begin{array}{lr}
 p \rightarrow  & i \\
\\
  q \rightarrow &  \\
   &  \\
  r \rightarrow & \\
&\\
  s \rightarrow &
  \end{array}
  \right|  \begin{array}{r}  \cr \cr j \cr  \cr \ov{j} \cr  \cr \ov{i} \cr \end{array}\shiftnoarrow{-1.25}{,} \quad \quad
 \left.\begin{array}{r}
i \\
\\
 j  \\
   \\
\ov{j} \\
\\
\\
  \end{array}
  \right|  \begin{array}{r}  \cr
   \cr
   \cr
\cr
\cr
\cr
\ov{i} \cr
  \end{array}\shiftnoarrow{-1.25}{.}$$

\end{enumerate}
We will denote the set of KN-symplectic tableaux of shape $\lambda$ with entries from $\mathcal{M}$ by $SP(\lambda,2m)$.
\begin{exa} The tableau $T=\vy{1 & 2 \cr \ov{3} & \ov{2} \cr \ov{2} & \ov{1} \cr}$ is not a KN-symplectic tableau since it violates the second property:  $2$ belongs to row $1$ and $\ov{2}$ belongs to row $2$ so  $p+q=3 > 2$.

On the other hand, $T=\vy{1 & 3 \cr \ov{3} & \ov{3} \cr \ov{2} & \ov{1} \cr}$  is a KN-symplectic tableau.
For property (2), $p+q=3 \leq 3$ and for property (3), $p=q=1$, $r=2$ and $s=3$ so $(q-p)+(s-r)=1<j-i=2$.

\end{exa}

The set $SP(\lambda,2m)$, where $\lambda$ is a partition of $n$, admits a crystal structure, which is given by embedding $SP(\lambda,2m)$ into the $n$-fold tensor power $\mathcal{B}(1)^{\otimes n}$ to give a bijection with a connected component of the crystal $\mathcal{B}(1)^{\otimes n}$.

The {\em column reading word} of a tableau $T \in SP(\lambda,2m)$ is the element $C(T) \in \mathcal{B}(1)^{\otimes n}$ given by reading the entries up columns from bottom to top, starting with the leftmost column.  To give a crystal structure on $SP(\lambda,2m)$, $T \in SP(\lambda, 2m)$ is identified with its image $C(T)$ in $\mathcal{B}(1)^{\otimes n}$ and the action of a crystal operator on $T$ is given by its action on $C(T)$.   The set of KN-tableaux $SP(\lambda,2m)$ is crystal isomorphic to the highest weight $U_q(\mathfrak{sp}(2m))$-crystal $\mathcal{B}(\lambda)$ with highest weight $\lambda$.   For details, see \cite[\S 6.3]{bump} or \cite[\S 8.3]{hongkang}.

If $T \in SP(\lambda,2m)$ and if $a_i$ (respectively $a_{\ov{i}}$) is equal to the number of entries equal to $i$ (respectively $\ov{i}$) in $T$, then the weight of $T$ is $\mbox{wt}(T)=(\chi_1,\ldots,\chi_m)$, where $\chi_i=a_i-a_{\ov{i}}$. Let $SP(\lambda,\chi)=\{T \in SP(\lambda,2m) \mid wt(T)=\chi\}$ and let $wt(SP(\lambda))$ denote the set of weights $\chi$ in $\Lambda$ for which there is a tableau $T \in SP(\lambda,2m)$ with $wt(T)=\chi$.

\begin{exa} \label{opeg} For $T=\vy{ 1 & 3 \cr \ov{3} & \ov{3} \cr \ov{2} & \ov{1} \cr} \in SP((2,2,2),6)$,
 $C(T)=\vy{\ov{2}\cr} \otimes \vy{\ov{3}\cr} \otimes \vy{1 \cr} \otimes \vy{\ov{1}\cr} \otimes \vy{\ov{3}\cr} \otimes \vy{3 \cr}$  and $\mbox{wt}(T)=(0,-1,-1)$.   We have $f_2(C(T))=\vy{\ov{2}\cr} \otimes \vy{\ov{3}\cr} \otimes \vy{1\cr} \otimes \vy{\ov{1}\cr} \otimes \vy{\ov{2}\cr} \otimes \vy{3 \cr} $
so $f_2(T)=\vy{1 & 3\cr \ov{3} & \ov{2} \cr \ov{2} & \ov{1} \cr}$.

\noindent Figure \ref{fig:crystalgraph1} gives an example of a crystal graph. \end{exa}

\bigskip

\begin{figure}[htpb]
\centering
\includegraphics[width=0.2\textwidth]{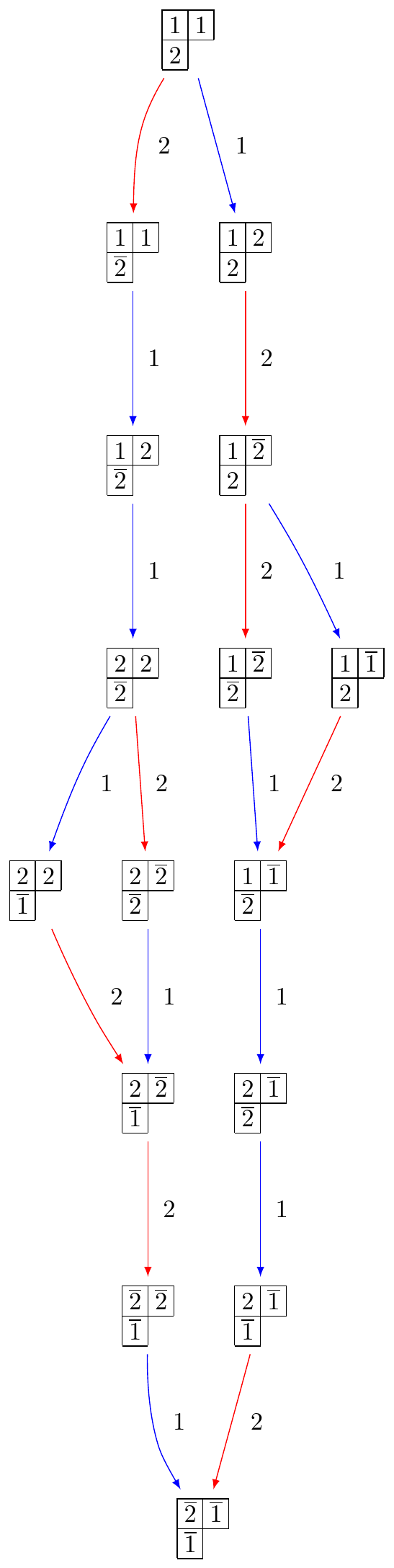}
\caption{Type $C_2$ crystal graph for $\lambda=(2,1)$.}
\label{fig:crystalgraph1}
\end{figure}

Given a simple reflection $s_i$ in the Weyl group $\mathcal{W}$,  $s_i(\chi)=\chi-( \chi,\alpha_i^\vee ) \alpha_i$ for $\chi \in \Lambda$, $i \in I$.  For type $C_m$, $s_i(\chi) =(\chi_1,\ldots,\chi_{i+1},\chi_i,\ldots,\chi_m)$ for $1 \leq i \leq m-1$ and $s_m (\chi)=(\chi_1,\ldots,\chi_{m-1},-\chi_m)$.

For $i \in I$,  define a bijection $\sigma_i$ on $\mathcal{B}(\lambda)$  by
\begin{equation} \label{sigmai} \sigma_i(b)=\begin{cases} f_i^k(b) & \mbox{if } k \geq 0 \cr
e_i^{-k}(b) & \mbox{if } k<0 \cr
\end{cases},\end{equation}
where $b \in \mathcal{B}(\lambda)$ and $k=(\mbox{wt}(b), \alpha_i^\vee)$.   The Weyl group $\mathcal{W}$ acts on $\mathcal{B}(\lambda)$ (\cite[Theorem 11.14]{bump}) by $s_i \cdot b =\sigma_i(b), \ b \in \mathcal{B}(\lambda).$
As well, we have (see \cite[Proposition 2.36]{bump}):
\begin{equation}\label{eqn1} \mbox{wt}(\sigma_i(b))=s_i(\mbox{wt}(b)).
\end{equation}
  Then $\sigma=\sigma_1 \sigma_2 \cdots \sigma_{m}$ gives a bijection on $\mathcal{B}(\lambda)$ and since the $\sigma_i$'s act on $\mathcal{B}(\lambda)$ as simple reflections of the Weyl group, $\sigma$ is a Coxeter element of $\mathcal{W}$ so has order equal to the Coxeter number of $\mathcal{W}$.  We summarize this in the following lemma.

  \begin{lem}\label{orderlem}Let  $\lambda$ be a partition and let $G=\langle \sigma \rangle$, where $\sigma=\sigma_1\cdots \sigma_m$. Then $G=\langle \sigma \rangle$ has order $2m$ and acts on $SP(\lambda,2m)$. \end{lem}

\begin{exa}\label{orbiteg} Consider the action of $G=\langle \sigma \rangle$ on $SP((2,1),6)$, where $\sigma=\sigma_1\sigma_2\sigma_3$.  There are two orbits of size two:  $$\left\{ \vy{1 & \ov{2} \cr 3 \cr}, \vy{2 & \ov{1} \cr \ov{3} \cr} \right\}, \ \left\{ \vy{1 & 3 \cr \ov{2} \cr}, \vy{2 & \ov{3} \cr \ov{1} \cr} \right\},$$ and 10 orbits of size six.
\end{exa}

If $T \in SP(\lambda,\chi)$, it follows from (\ref{eqn1}) that for $1 \leq i \leq m-1$, $$\mbox{wt}(\sigma_i T)=(\chi_1,\ldots,\chi_{i+1},\chi_i,\ldots,\chi_m) \ \ \mbox{and} \ \ \mbox{wt}(\sigma_m T)=(\chi_1,\ldots,\chi_{m-1},-\chi_m).$$  This gives the following lemma, which will be useful throughout the paper.

\begin{lem}\label{wtshiftlem} If $T \in SP(\lambda,2m)$ and $\mbox{wt}(T)=(\chi_1,\ldots,\chi_m)$ then $$\mbox{wt}(\sigma T)=(-\chi_m,\chi_1,\ldots,\chi_{m-1}).$$ \end{lem}

Given two partitions $\mu$ and $\lambda$, $\mu \subseteq \lambda$ if $\mu_i \leq \lambda_i$ for all $i$.  The skew shape $\lambda/\mu$ is obtained by removing the boxes of the Young diagram of shape $\mu$ from that of shape $\lambda$ and $\vert \lambda/\mu \vert =\vert \lambda \vert - \vert \mu \vert$.   The set of KN-skew symplectic tableaux of shape $\lambda/\mu$ with entries from the set $1<2<\cdots <m<\ov{m} < \cdots <\ov{1}$ will be denoted $SP(\lambda/\mu,2m)$.  Their entries satisfy the conditions  defined in \cite[\S 2.2]{santos} (see also \cite[\S 6]{lecouvey}) and, when $\mu=\emptyset$, we obtain  $SP(\lambda, 2m)$.  Our proofs in Section $5$ will refer to the weights of KN-skew symplectic tableaux, which are defined in the same way as for KN-symplectic tableaux.   We let $SP(\lambda/\mu,\chi)=\{ T \in SP(\lambda/\mu) \mid wt(T)=\chi\}$ and $wt(SP(\lambda/\mu))$ is the set of $m$-tuples $\chi$ for which there is a $T \in SP(\lambda/\mu, 2m)$ with $wt(T)=\chi$.  There is also a crystal structure on $SP(\lambda/\mu, 2m)$ (see \cite[\S 6]{lecouvey}).  Lemma \ref{wtshiftlem}, and the discussion preceding it, also apply in this setting.

\section{The symplectic hook-content formula}\label{sec:hooks}
Let $T \in SP(\lambda,2m)$ where $wt(T)=(\chi_1,\ldots,\chi_m)$ with $\chi_i=a_i-a_{\ov{i}}$, where $a_i$ records the number of $i$'s in $T$ and $a_{\ov{i}}$ the number of $\ov{i}$'s in $T$.
Define \begin{equation}\label{eqpwr} pwr(T)=\sum_{i=1}^m ((i-1)a_i+(2m-i)a_{\ov{i}}).\end{equation}  Then $q^{pwr(T)}$ is the product given by assigning $q^{i-1}$ to each $i \in \{1,\ldots,m\}$ in $T$ and $q^{2m-i}$ to each entry $\ov{i} \in \{ \ov{1},\ldots,\ov{m}\}$ in $T$.

\begin{lem}\label{spspec} Let $\lambda$ be a partition of $n$ and let $T\in SP(\lambda,\chi)$. Then $$pwr(T)=\sum_{i=1}^m (i-1) \chi_i + \frac{2m-1}{2}\left(n-\sum_{i=1}^m \chi_i \right).$$ 

\end{lem}

\begin{proof} Let $wt(T)=\chi=(\chi_1,\ldots,\chi_m)$, where $\chi_i=a_i-a_{\ov{i}}$. Since \begin{equation}\label{wtsum} \sum_{i=1}^m (a_i + a_{\ov{i}})=n \mbox{ and }  \sum_{i=1}^m (a_i-a_{\ov{i}})=\sum_{i=1}^m \chi_i, \end{equation}
\begin{eqnarray*} pwr(T)&=&\sum_{i=1}^m i \chi_i-\frac{1}{2}\left(n+\sum_{i=1}^m \chi_i\right)+m\left(n-\sum_{i=1}^m \chi_i\right) \cr
&=& \sum_{i=1}^m (i-1) \chi_i + \frac{2m-1}{2}\left(n-\sum_{i=1}^m \chi_i \right).\cr \end{eqnarray*}


\end{proof}

The above lemma shows that $pwr(T)$ is completely determined by the weight of $T$.  We will also use the notation $pwr(\chi)$ to denote $pwr(T)$, where $wt(T)=\chi$.  In light of the lemma, we have the following:
$$\sum_{T \in SP(\lambda,2m)}q^{pwr(T)}=\sum_{\chi \in wt(SP(\lambda))}\vert SP(\lambda,\chi) \vert q^{pwr(\chi)}.$$

\begin{exa} (1) For $T \in SP((2,2,2),6)$ as in Example \ref{opeg}, $pwr(T)=17$.

\bigskip

\noindent (2) Let $\lambda=(2,1)$ and $m=2$.  Referring to Figure \ref{fig:crystalgraph1}, $\vert SP(\lambda,2m) \vert =16$ and
\begin{eqnarray*} \sum_{T \in SP(\lambda,2m)} q^{pwr(T)}&=&q+2q^2+2q^3+3q^4+3q^5+2q^6+2q^7+q^8\cr& \equiv& 4(1+q+q^2+q^3) \mod q^4-1.\end{eqnarray*}This polynomial is of the type covered by Theorem \ref{equivthm}.\end{exa}

\bigskip

The {\em hook length} of a box in the $i$th row and $j$th column of the Young diagram of shape $\lambda$ is the number of boxes in its hook.  In other words, $h(i,j)=\lambda_i+\lambda_j^t-i-j+1$.  Define
\begin{equation*}
r_\lambda(i,j) = \left\{
\begin{array}{ll}
\lambda_i+\lambda_j-i-j+2    & \text{if $i>j$,}\\
i+j-\lambda_i^t-\lambda_j^t  & \text{if $i \leq j$.}
\end{array}\right.
\end{equation*}

\noindent The hook-content formula for symplectic tableaux \cite[Corollary 4.6]{campbellstokke} is given by
\begin{equation}\label{hookcontent}\vert SP(\lambda,2m)\vert=\prod_{(i,j) \in [\lambda]}
\frac{2m+r_\lambda(i,j)} { h(i,j) }.\end{equation}

\noindent For staircase tableaux of the form $\lambda=(m,m-1,\ldots,1)$, the above formula simplifies nicely \cite[Corollary 4.48]{bernsteinstriker}.  

For $T \in SP(\lambda,\chi)$,  define $x^{wt(T)}=x_1^{\chi_1}x_2^{\chi_2} \cdots x_m^{\chi_m}=\prod_{i=1}^m x_i^{a_i-a_{\bar{i}}}$.   The symplectic Schur function is the character of the irreducible $\mathfrak{sp}(2m)$-representation with highest weight $\lambda$, defined as $$ sp_{\lambda}(x_1^{\pm 1},\ldots,x_m^{\pm 1})=\sum_{T \in SP(\lambda,2m)} x^{wt(T)}.$$
In \cite{campbellstokke}, we worked with a specialization $sp_\lambda(q,q^3,q^5,\ldots,q^{2m-1})$ to prove (\ref{hookcontent}), but this polynomial does not work as a CSP polynomial with the action under consideration.  Instead, we will use a natural $q$-analogue of (\ref{hookcontent}) as our CSP polynomial.

Using \cite[Equation~24.18]{fultonharris}, with $x_j=q^{j-1}$ and $x_j^{-1}=q^{2m-j}$, we can express $\sum q^{pwr(T)}$ as a quotient of determinants:
\begin{equation}\label{spdeterm}
\sum_{T \in SP(\lambda,2m)} q^{pwr(T)}=
\frac{\vert q^{(j-1)(\lambda_i+m-i+1)}-q^{(2m-j)(\lambda_{i}+m-i+1)}\vert_{i,j=1}^{m}}{\vert q^{(j-1)(m-i+1)}-q^{(2m-j)(m-i+1)} \vert_{i,j=1}^{m}}.
\end{equation}
Here we take $\lambda_i=0$ when $i>\ell(\lambda)$.  For a positive integer $k$ define
$[k]=1-q^k$ and let $[k]!=[k][k-1] \cdots [1]$. (Note that in \cite{campbellstokke} we worked with $\langle k \rangle = q^k-q^{-k}$.)

\begin{lem}\label{firstqlem} Let $\lambda$ be a partition and  let $\mu_i=\lambda_i+m-i$.  Then $$\vert q^{(j-1)(\mu_i+1)}-q^{(2m-j)(\mu_i+1)}\vert_{i,j=1}^m =\ds q^{\sum\limits_{i=1}^m (i-1)(\mu_i+1)} \prod_{i=1}^m [\mu_i+1]\prod_{1 \leq i <j \leq m} [\mu_i-\mu_j][\mu_i+\mu_j+2].$$
 \end{lem}
 \begin{proof} Using elementary row operations,  $\vert q^{(j-1)(\mu_i+1)}-q^{(2m-j)(\mu_i+1)}\vert_{1 \leq i,j \leq m}$  equals
 \begin{eqnarray*} && (-1)^{m(m-1)/2} \prod_{i=1}^m (1-q^{\mu_i+1})\vert q^{(m-j)(\mu_i+1)} (1+q^{\mu_i+1})^{2j-2} \vert \cr
 &=& (-1)^{m(m-1)/2} \prod_{i=1}^m (1-q^{\mu_i+1}) q^{(m-1)(\mu_i+1)} \vert q^{-(j-1)(\mu_i+1)}(1+q^{\mu_i+1})^{2j-2} \vert \cr
 &=&(-1)^{m(m-1)/2}q^{\sum\limits_{i=1}^m (m-1)(\mu_i+1)} \prod_{i=1}^m (1-q^{\mu_i+1}) \vert(q^{-(\mu_i+1)}(1+q^{\mu_i+1})^2)^{j-1} \vert. \cr \end{eqnarray*}
The Vandermonde determinant $\vert(q^{-(\mu_i+1)}(1+q^{\mu_i+1})^2)^{j-1} \vert$ is equal to
 \begin{eqnarray*} & & \prod_{1 \leq i<j \leq m} (q^{-(\mu_j+1)}(1+q^{\mu_j+1})^2-q^{-(\mu_i+1)}(1+q^{\mu_i+1})^2)\cr
 &=& \prod_{1 \leq i<j \leq m} -q^{-(\mu_i+1)}(1+q^{2\mu_i+2}-q^{\mu_i-\mu_j}-q^{\mu_i+\mu_j+2})\cr
 &=& (-1)^{\frac{m(m-1)}{2}}q^{-\sum\limits_{i=1}^m (m-i)(\mu_i+1)}\prod_{1 \leq i <j \leq m}  (1-q^{\mu_i-\mu_j})(1-q^{\mu_i+\mu_j+2}) \cr
 &=& (-1)^{\frac{m(m-1)}{2}} q^{-\sum\limits_{i=1}^m(m-i)(\mu_i+1)} \prod_{1 \leq i<j \leq m}[\mu_i-\mu_j][\mu_i+\mu_j+2]. \cr \end{eqnarray*}
Substituting, we obtain the result.  \end{proof}

 \begin{cor}\label{secondqlem}We have $$\vert q^{(j-1)(m-i+1)}-q^{(2m-j)(m-i+1)}\vert_{1 \leq i,j \leq m}=q^{\sum\limits_{i=1}^m (i-1)(m-i+1)}\prod_{i=1}^m [2i-1]!$$

 \end{cor}
 \begin{proof}  This follows from Lemma \ref{firstqlem} by taking $\lambda=\emptyset$.  We have 
 \begin{eqnarray*} \vert q^{(j-1)(m-i+1)}-q^{(2m-j)(m-i+1)} \vert &=&q^{\sum\limits_{i=1}^m (i-1)(m-i+1)} \prod_{i=1}^m[m-i+1]\prod_{1 \leq i <j \leq m} [j-i][2m-j-i+2] \cr
 &=& q^{\sum\limits_{i=1}^m (i-1)(m-i+1)} \prod_{i=1}^m [2i-1]! \cr \end{eqnarray*}\end{proof}

 Define $f_{sp}^\lambda(q)$ to be the $q$-analogue of the symplectic hook-length formula:
 $$f_{sp}^\lambda(q)=\prod\limits_{(i,j) \in [\lambda]} \frac{[2m+r_{\lambda}(i,j)]}{[h(i,j)]}.$$ 

\begin{thm} \label{hookthm} Let $\lambda$ be a partition and let $\kappa(\lambda)=\sum (i-1)\lambda_i$.  Then
$$f_{sp}^\lambda(q)=q^{-\kappa(\lambda)}\sum_{T \in SP(\lambda,2m)}q^{pwr(T)}.$$ \end{thm}

\begin{proof}
By \cite[7.101]{stanleybook} $\ds \prod\limits_{(i,j) \in [\lambda]} [h_{\lambda}(i,j)]=\frac{\prod\limits_{i=1}^m [\mu_i]!}{\prod\limits_{1 \leq i<j \leq m}[\mu_i-\mu_j]}$.   The proof of \cite[Lemma 4.3]{campbellstokke} yields $$\prod\limits_{(i,j)\in [\lambda]} [2m+r_{\lambda}(i,j)]=\prod\limits_{i=1}^m\frac{[\mu_i+1]!}{[2i-1]!}\prod\limits_{1\leq i<j \leq m} {[\mu_i+\mu_j+2]}.$$  Since $\sum_{i=1}^m(i-1)(\mu_i+1)-\sum_{i=1}^m (i-1)(m-i+1)=\kappa(\lambda)$,
\begin{eqnarray*}\sum_{T \in SP(\lambda,2m)}q^{pwr(T)} &=& q^{\kappa(\lambda)}\frac{\prod_{i=1}^m [\mu_i+1]\prod\limits_{1 \leq i<j \leq m}[\mu_i-\mu_j][\mu_i+\mu_j+2]}{\prod\limits_{i=1}^m [2i-1]!}\cr
&=& q^{\kappa(\lambda)}\frac{\prod_{i=1}^m [\mu_i+1]!\prod\limits_{1 \leq i<j \leq m}[\mu_i-\mu_j][\mu_i+\mu_j+2]}{\prod\limits_{i=1}^m[\mu_i]!\prod\limits_{i=1}^m [2i-1]!}\cr
&=& q^{\kappa(\lambda)} \prod_{(i,j) \in [\lambda]} \frac{[2m+r_\lambda(i,j)]}{[h(i,j)]}=q^{\kappa(\lambda)}f_{sp}^\lambda(q).
\end{eqnarray*} 
\end{proof}

\section{A cyclic sieving phenomenon for symplectic tableaux}\label{sec:main}

We will prove the results in this section for the set $SP(\lambda/\mu,2m)$, with a cyclic group action that shifts weights cyclically as in Lemma \ref{wtshiftlem}.  As a corollary, we obtain a CSP for $SP(\lambda,2m)$ with action induced by the $U_q(\mathfrak{sp}(2m))$-crystal structure. Our proofs rely on properties of symplectic weights.  It follows from (\ref{wtsum}) that if $\vert \lambda/\mu\vert =n$ and $T \in SP(\lambda/\mu,\chi)$, then $\sum_{i=1}^m \chi_i=n-2\ell$ for some $0 \leq \ell \leq n$.  Thus, if $n$ is odd, $\sum_{i=1}^m \chi_i$ is odd, which is a fact we will refer to in our proofs.

\begin{thm}\label{orderthm} Let $\vert \lambda/\mu \vert=n$, where $n$ is odd, and suppose that  $gcd(m,p)=1$ for any odd prime $p$ with $p\leq n$.  Let $G=\langle g \rangle$ be a cyclic group of order $2m$ with the property that $wt(g T)=(-\chi_m,\chi_1, \ldots,\chi_{m-1})$ for all $T \in SP(\lambda/\mu,\chi)$.  Then every orbit in $SP(\lambda/\mu,2m)$ under the action of $G$ has cardinality $2m$.
\end{thm}
\begin{proof}
Let  $T \in SP(\lambda/\mu, \chi)$.  Since $gcd(m,p)=1$ for odd primes $p \leq n$, either $m=2^k$ for some positive integer $k$ or $m>n$ and $m$ is not divisible by any odd prime $p\leq n$.  

Suppose that the orbit of $T$ has fewer than $2m$ elements.
  Then, for the case $m=2^k$, $g^a(T)=T$, for some $a=2^j$ with $0 \leq j \leq k$.  Since $\mbox{wt}(T)=\mbox{wt}(g^a T)$,
  $$(\chi_1,\chi_2, \ldots, \chi_m)=(-\chi_{m-a+1}, -\chi_{m-a+2}, \ldots, -\chi_m, \chi_1, \cdots, \chi_{m-a}).$$
  Since $a$ divides $m$, $\chi_i=\chi_{i+a}=\chi_{i+2a}=\cdots=\chi_{m-a+i}=-\chi_i$ for $1 \leq i \leq m$ and, in the  case where $a=m$, we have $(-\chi_1,-\chi_2,\ldots,-\chi_m)=(\chi_1,\chi_2,\ldots,\chi_m)$, again yielding $\chi_i=-\chi_i$.  But then $\chi_i=0$ for $1 \leq i \leq m$, which is not possible.

Now suppose  $m \neq 2^k$ and that $gcd(m,p)=1$ for any odd prime $p$, with $p\leq n$. If $g^a(T)=T$, where $a$ divides $m$, the argument is the same as above.  The other possibility is that $g^{2a}(T)=T$, where $a$ divides $m$, $1 \leq a <m$ and $2a$ does not divide $m$.  Then $m=ba$, where $b=2k+1$ is odd. Thus $m=2ak +a \equiv a~\mbox{mod} ~ 2a$ and
$$(\chi_1,\chi_2,\ldots,\chi_m)=(-\chi_{m-2a+1},\ldots, -\chi_m, \chi_1, \ldots, \chi_{m-2a})=wt(g^{2a}(T)).$$
Then $\chi_i=\chi_{i+2a}=\chi_{i+4a}=\cdots = \chi_{i+2ka}=-\chi_{i+a}=-\chi_{i+3a}=\cdots=\chi_i$, for $1 \leq i \leq a-1$ and $\chi_a=\chi_{3a}=\cdots=-\chi_{2a}=-\chi_{4a}=\cdots=\chi_a$ so $$wt(T)=(\chi_1,\ldots,\chi_a,-\chi_1,\ldots,-\chi_a,\chi_1,\ldots,\chi_a,\ldots).$$ Since $m=ab$, there are a total of $b$ entries in the $m$-tuple that are equal to $\chi_i$ or $-\chi_i$, for each $1 \leq i \leq a$.   Since $\sum_{i=1}^m \chi_i \neq 0$, $\chi_i \neq 0$ for some $1\leq i \leq a$.  Assuming $\chi_i >0$, $T$ contains at least $\chi_i$ entries equal to $i$, at least $\chi_i$ entries equal to $\ov{i+a}$, et cetera.  But then $T$ has at least $b$ entries so $n \geq b$.   Since $b$ is odd  and no odd prime less than $n$ divides $m$ this is impossible.  \end{proof}

Lemmas \ref{orderlem} and \ref{wtshiftlem} yield the following corollary.

\begin{cor}\label{ordercor} Let $\lambda$ be a partition of $n$, where $n$ is odd, and suppose that  $gcd(m,p)=1$ for any odd prime $p$ with $p\leq n$.  Then every orbit in $SP(\lambda,2m)$ under the action of $G=\langle \sigma \rangle$ has cardinality $2m$.
\end{cor}
\begin{rem} If $\lambda$ is a partition of an even number then $SP(\lambda,2m)$ may have single-element orbits under the action of $\langle \sigma \rangle$.  As well, if $\lambda$ is a partition of an odd number $n$ and $m$ is divisible by some prime $p \leq n$, then $SP(\lambda,2m)$ may have orbits with fewer than $2m$ elements  (see Example \ref{orbiteg}). It is also worth pointing out that, while our hypotheses guarantee that $\mbox{gcd}(m,n)=1$, this is not sufficient.  For example, if $\lambda=(4,1)$ and $m=6$, there are orbits with fewer than $12$ elements.
\end{rem}

For $T \in SP(\lambda/\mu,2m)$ define $pwr(T)$ as in (\ref{eqpwr}) and define
$$\displaystyle X(q)=\sum_{T \in SP(\lambda/\mu,2m)} q^{pwr(T)}=\sum_{\chi \in wt(SP(\lambda/\mu))} \vert SP(\lambda/\mu,\chi)\vert q^{pwr(\chi)}.$$
A set of integers $S$ is a {\em complete residue system} modulo a positive integer $n$ if $\vert S \vert =n$ and no two elements in $S$ are congruent modulo $n$.  In order to prove a CSP for $SP(\lambda/\mu,2m)$, using $X(q)$ as a CSP polynomial, we aim to partition the set of weights $wt(SP(\lambda/\mu))$ into sets $\mathcal{A}_\chi$ of cardinality $2m$ such that the powers in the polynomial $X(q)$ associated to each $\mathcal{A}_\chi$ form a complete residue system modulo $2m$.   The symplectic version of \cite[Lemma 3.2]{ohpark1} does not hold.  Instead, we will work with an action of signed permutations associated with the dihedral group of order $2m$ on $wt(SP(\lambda/\mu))$  to divide the powers into sets that each form complete residue systems modulo $2m$.

The symmetric group $\mathfrak{S}_m$ acts on $wt(SP(\lambda/\mu))$ by $\theta \chi=(\chi_{\theta^{-1}(1)}, \ldots, \chi_{\theta^{-1}(m)})$, and \begin{equation}\label{symact} \vert SP(\lambda/\mu,\chi) \vert = \vert SP(\lambda/\mu, \theta \chi)\vert=\vert SP(\lambda/\mu,-\theta \chi) \vert,  \ \theta \in \mathfrak{S}_m, \ \chi \in wt(SP(\lambda/\mu)).\end{equation}   For $\mu=\emptyset$ this is well-known, since $\vert SP(\lambda,\chi)\vert$  is equal to the dimension of the corresponding weight space for the irreducible $\mathfrak{sp}(2m)$-representation with highest weight $\lambda$. For skew tableaux, this can be seen using the bijections (\ref{sigmai}), and their impact on weights (\ref{eqn1}), and a symplectic version of  the argument in \cite[(3.1)]{alexander}.

\begin{lem}\label{reslem} Let $\vert \lambda/\mu\vert =n,$ where $n$ is odd, and suppose that $gcd(m,p)=1$ for any odd prime $p\leq n$.    Let $D_{2m}=\langle \gamma=(1,2,\ldots,m), \beta=(2,m)(3,m-1)\cdots \rangle$.  If $\chi \in wt(SP(\lambda/\mu))$, then the set
$\{pwr(\gamma^t\chi), pwr(-\gamma^t\beta\chi) \mid 0 \leq t \leq m-1\}$ is a complete residue system modulo $2m$.\end{lem}
\begin{proof}For $0 \leq t \leq m-1$, $\ds pwr(\gamma^t \chi)=\sum_{i=1}^m ((i-1+t)~ \mbox{mod} ~ m) \chi_i + \frac{2m-1}{2}(n-\sum_{i=1}^m \chi_i)$.
Since $gcd(m,p)=1$ for odd primes $p\leq n$ and $\sum_{i=1}^m \chi_i$ is odd with $\sum_{i=1}^m \chi_i=n-2 \ell$, where $0 \leq \ell \leq n$, $gcd(\sum_{i=1}^m \chi_i,m)=1$.  It follows that the set $\{pwr(\gamma^t \chi) \mid 0 \leq t \leq m-1\}$ is a complete residue system modulo $m$, so no two elements in the set $\{pwr(\gamma^t \chi) \mid 0 \leq t \leq m-1\}$ are congruent modulo $2m$.

Also $pwr(-\gamma^t\beta \chi)\equiv pwr(-\gamma^{t+1}\beta \chi)+\sum_{i=1}^m \chi_i \mod m,$
so $\{pwr(-\gamma^t\beta \chi)\mid 0 \leq t \leq m-1\}$ is a complete residue set modulo $m$.  Thus none of the elements in  $\{pwr(-\gamma^t\beta \chi) \mid 0 \leq t \leq m-1\}$ are congruent modulo $2m$.  Lastly, $$pwr(\chi)- pwr(-\gamma^t\beta \chi)=(t+1)\sum_{i=1}^m \chi_i-2m\sum_{i=1}^m \chi_i+m\sum_{i=t+2}^m \chi_i,$$
so if $pwr(\chi)- pwr(-\gamma^t\beta \chi)$ is divisible by $2m$, then $m$ divides $(t+1) \sum_{i=1}^m \chi_i$.  Then, since $\mbox{gcd}(m, \sum_{i=1}^m \chi_i)=1$, $m$ divides $t+1$, which is only possible if $m=t+1$.  However, if $m=t+1$, then $pwr(\chi)- pwr(-\gamma^t\beta \chi)=-m\sum_{i=1}^m \chi_i$ and, since $\sum_{i=1}^m \chi_i$ is odd, this cannot be divisible by $2m$, so $pwr(\chi)\not\equiv pwr(-\gamma^t\beta \chi) \mod 2m$ for $0 \leq t \leq m-1$.  It follows that $pwr(\gamma^{t_1} \chi) \not\equiv  pwr(-\gamma^{t_2}\beta \chi) \mod 2m$, for $t_1 \neq t_2$ so the set in question is a complete residue set modulo $2m$.
\end{proof}

 Given  $\chi \in wt(SP(\lambda/\mu))$, define the subset $\mathcal{A}_\chi$ of $wt(SP(\lambda/\mu))$ as
$\mathcal{A}_{\chi}=\{\gamma^t\chi, -\gamma^t\beta\chi \mid 0 \leq t \leq m-1\}.$
Under the conditions of Lemma \ref{reslem},  $\vert \mathcal{A}_\chi\vert =2m$.   A routine argument shows that whenever $\mathcal{A}_{\chi_1} \cap \mathcal{A}_{\chi_2} \neq \emptyset$, we have $\mathcal{A}_{\chi_1} = \mathcal{A}_{\chi_2}$.  Thus the sets $\mathcal{A}_\chi$ partition $wt(SP(\lambda/\mu))$ into sets of size $2m$.

\begin{exa} (1) Let $\lambda=(2,1)$, $m=4$ and $\chi=(2,-1,0,0)$.  Then
$$\mathcal{A}_\chi = \{(2,-1,0,0), (0,2,-1,0),(0,0,2,-1),(-1,0,0,2),(-2,0,0,1),(1,-2,0,0),(0,1,-2,0),(0,0,1,-2)\},$$ which is
covered by Lemma \ref{reslem} and the corresponding powers form a complete residue set modulo $8$.

\bigskip

\noindent (2) If the hypotheses in Lemma \ref{reslem} are relaxed, the result does not hold.  For example, if $\lambda=(2,1)$, $m=3$ and $\chi=(1,1,1,0,0,0)$ then $\{pwr(\gamma^t\chi), pwr(-\gamma^t\beta\chi) \mid 0 \leq t \leq 2\}$ is not a complete residue set modulo $6$.  Note that, in this case, $f_{sp}^\lambda(q)\equiv 10q^5+11q^4+11q^3+10q^2+11q+11 \mod q^6-1$ and Corollary \ref{maincor} does not hold.

\end{exa}

 \begin{thm} \label{equivthm} Let $\vert \lambda/\mu \vert =n$ where $n$ is odd, and suppose that $gcd(m,p)=1$ for any odd prime $p$ with $p\leq n$. Then $$X(q) \equiv \frac{\vert SP(\lambda/\mu, 2m) \vert}{2m}  \sum_{k=0}^{2m-1} q^k \mod q^{2m}-1.$$
 \end{thm}
 \begin{proof}   Let $\mathcal{A} \subseteq wt(SP(\lambda/\mu))$ be a transversal for the collection of sets $\mathcal{A}_\chi$.  By (\ref{symact}), $\vert SP(\lambda/\mu, \xi)\vert = \vert SP(\lambda/\mu, \chi)\vert$ for all $\xi \in \mathcal{A}_\chi$.  Then
 $$\ds \vert SP(\lambda/\mu, 2m) \vert = \sum_{\chi\in \mathcal{A}} \vert SP(\lambda/\mu, \chi) \vert \vert A_\chi\vert=2m \sum_{\chi\in \mathcal{A}} \vert SP(\lambda/\mu,\chi) \vert .$$
By Lemma \ref{reslem}, \begin{eqnarray*}X(q)&=&\sum_{\chi\in wt(SP(\lambda/\mu))} \vert SP(\lambda/\mu,\chi) \vert q^{pwr(\chi)} \cr
 &=& \sum_{\chi \in \mathcal{A}} \vert SP(\lambda/\mu,\chi) \vert \sum_{\xi \in \mathcal{A}_\chi} q^{pwr(\xi)} \cr
&\equiv& \sum_{\chi \in \mathcal{A}} \vert SP(\lambda/\mu,\chi) \vert (1 + q + \cdots + q^{{2m}-1})\mod q^{2m} -1 \cr
 &=& \frac{\vert SP(\lambda/\mu,2m) \vert}{2m} (1 + q + \cdots + q^{2m-1})\mod q^{2m} -1.\cr
 \end{eqnarray*}  \end{proof}

 \noindent Combining Theorem \ref{orderthm} and Theorem \ref{equivthm}, we obtain the following CSP.
 \begin{thm}\label{cspthm}  Let $\vert \lambda/\mu \vert=n$, where $n$ is odd, and suppose that  $gcd(m,p)=1$ for any odd prime $p \leq n$.  Let $\langle g \rangle$ be a cyclic group of order $2m$ that acts on $SP(\lambda/\mu,2m)$ with the property that $wt(g T)=(-\chi_m,\chi_1, \ldots,\chi_{m-1})$ for all $T \in SP(\lambda/\mu,\chi)$.    Then $(SP(\lambda/\mu,2m),\langle g \rangle,X(q))$ exhibits the cyclic sieving phenomenon. \end{thm}

\begin{cor}\label{maincor}Let $\lambda$ be a partition of $n$, where $n$ is odd, and suppose that $gcd(m,p)=1$ for any odd prime $p \leq n$.  Then $(SP(\lambda,2m),\langle \sigma \rangle,f^\lambda_{sp}(q))$ exhibits the cyclic sieving phenomenon, where $\sigma$ is the natural action induced by the $U_q(\mathfrak{sp}(2m))$-crystal and $f^\lambda_{sp}(q)$ is the  $q$-analogue of the symplectic hook-content formula. \end{cor}

\noindent {\bf Acknowledgement.}  The authors wish to thank two anonymous referees for suggestions that improved the paper, including a suggestion to generalize the main result to skew shapes.

\end{document}